\documentclass[12pt]{amsart}
\usepackage[a4paper, total={6in, 8in}]{geometry}
\usepackage{cmap}
\usepackage{amsmath,amsthm,amssymb}
\usepackage[T2A]{fontenc}
\usepackage[english]{babel}
\usepackage[cp1251]{inputenc}
\usepackage{amscd}

\usepackage{MnSymbol} 
\usepackage{multirow}  

\usepackage{hyperref}

\theoremstyle{plain}
\newtheorem{corollary}{Corollary}
\newtheorem{assertion}{Assertion}
\newtheorem{theorem}{Theorem}

\newtheorem{lemma}{Lemma}

\theoremstyle{definition}
\newtheorem{definition}{Definition}
\newtheorem{remark}{Remark}
\newtheorem{assumption}{Assumption}

\begin{document}

\title{Invariant foliations of non-degenerate bi-Hamiltonian structures}
\author{I.\,K.~Kozlov}
\thanks{This work was supported by the Russian Foundation for Basic Research (grant no. 13-01-00664a) and the program ``Leading Scientific Schools'' (grant no. NSh-581.2014.1)}
\address{No Affiliation, Moscow, Russia. \ {\bf Email:}  ikozlov90@gmail.com}
\date{}

\begin{abstract}  In this paper, we describe all invariant distributions of non-degenerate bi-Hamiltonian structures and investigate their integrability in the  neighbourhood of a generic point. \end{abstract}

\maketitle

\section{Introduction and main results}

In this paper, we investigate the integrability of certain distributions, which naturally arise when considering
pairs of compatible non-degenerate Poisson brackets on real and
complex manifolds.

\begin{assumption} In the real case, all objects in this paper (manifolds, differential forms, etc.) are assumed to be smooth (of class $C^{\infty} $). In the complex case, all objects under consideration are complex analytic.  \end{assumption}

\begin{definition}  A pair of differential  $2$-forms $(\omega_0,
\omega_1)$  on a manifold $M$ is called
\emph{compatible} if the following holds:

\begin{enumerate}

\item The form $\omega_0$ is non-degenerate.

\item Both forms $\omega_0$ and $\omega_1$ are closed: \[d\omega_0 = 0 , \qquad d \omega_1 =0. \]

\item The Nijenhuis tensor of the field of endomorphisms $P = \omega_{0}^{-1} \omega_{1}$
is equal to zero: \[N_P =0. \]
\end{enumerate} \end{definition}

Recall that the Nijenhuis tensor $N_P$ of a field of endomorphisms $P$ is  given by the formula:
\[N_P (X, Y) = [PX, PY] - P [PX, Y] - P[X, PY] + P^2 [X, Y], \]
for any vector fields $X$ and $Y$.

Two symplectic structures on a manifold $M$ are compatible if and only if the corresponding
Poisson brackets are compatible (i.e. they define a bi-Hamiltonian structure).  Thus a pair of compatible $2$-forms is also called \textbf{a non-degenerate bi-Hamiltonian structure}. For more about bi-Hamiltonian structures see, e.g.,\cite{BolsinovIzosimomKonyaevOshemkov12} and the references therein.

In this paper we examine the invariant distributions, which are defined as follows.

\begin{definition}  Let $A$ and $B$ be two skew-symmetric bilinear forms on a linear
space $V$. We call a subspace $W$ of the linear space $V$
\textbf{invariant} if it is invariant with respect to the action of the group
of automorphisms $\textnormal{Aut}(V, A, B)$ consisting of all linear transformations which preserve both forms $A$ and $B$. \end{definition}

Sometimes, instead of a pair of forms on a linear space we will consider a pair consisting of a bilinear form  $B$  and an operator $P$. Here we identify the pair $(B, P)$  with the pair of bilinear forms $(B \circ P, B)$.

\begin{definition} Consider a  pair of compatible $2$-forms $(\omega_0, \omega_1)$ on a
manifold $M$. We call a distribution $F$ on the manifold $M$ \emph{invariant}
if every subspace $F_x$  is an invariant subspace of the corresponding
the tangent space $(T_x M, \omega_0, \omega_1)$.
\end{definition}

In this paper, we investigate the integrability of invariant distributions in the  neighbourhood of a generic point. The problem of
integration of invariant distributions was posed in \cite{BolsinovIzosimomKonyaevOshemkov12}. On the one hand, this problem can be viewed as a question about the local structure of compatible Poisson brackets and as a continuation of the research of I.\,S.~Zakharevich \cite{Zakharevich01},  A.~Panasyuk \cite{Panasyuk00} and  F.~Turiel \cite{Turiel94, Turiel10, Turiel11}.  On the other hand, this problem is related to the question of integrability of bi-Hamiltonian systems and the generalized Mishchenko--Fomenko hypothesis about existence of a set of functions in bi-involution for the pencil of argument shift method on Lie coalgebras (see \cite{BolsinovIzosimomKonyaevOshemkov12}). One of the most effective current methods of constructing integrals for bi-Hamiltonian systems on Lie algebras --- the argument shift method --- was developed by A.\,S.~Mishchenko and A.\,T.~Fomenko in the papers \cite{MishchenkoFomenko78GeneralizedLiouville, MishchenkoFomenko78EulerEquations}. Nevertheless, the prove of the generalized Mishchenko--Fomenko hypothesis requires some additional considerations ---  the criterion for the completeness of a family of functions that is constructed by the argument shift method  was obtained by A.\,V.~Bolsinov in \cite{Bolsinov88, Bolsinov92} (for more details about the argument shift method and its possible generalizations see also \cite{Trofimov95}). Therefore, in this paper we try to approach the question of integrability of bi-Hamiltonian systems from the other side. Namely, in this paper we search for foliations that can be described in terms of a bi-Hamiltonian structure or, more generally, for integrable distributions that are naturally associated with a bi-Hamiltonian structure.

Unfortunately, at this moment, singularities of bi-Hamiltonian structures are still virtually not investigated. Thus we restrict ourselves to a certain class of generic points (these points will be called regular). In order to define the set of
points of non-degenerate bi-Hamiltonian structures under consideration we need the following two theorems from linear algebra.

\begin{theorem}[Jordan--Kronecker theorem]
\label{T:Jordan-Kronecker_theorem} 
Let  $A$ and $B$ be two skew-symmetric bilinear forms on a finite-dimensional vector
space $V$ over a field $\mathbb{K}$.  If the field  $\mathbb{K}$ is algebraically closed, then there exists a basis in the space $V$ such that the matrices of both forms  $A$ and $B$ are block-diagonal
matrices:

\begin{equation}\label{Eq:Jordan-Kronecker-form} {\footnotesize A =
\begin{pmatrix}
A_1 &     &        &      \\
    & A_2 &        &      \\
    &     & \ddots &      \\
    &     &        & A_k  \\
\end{pmatrix}
\quad  B=
\begin{pmatrix}
B_1 &     &        &      \\
    & B_2 &        &      \\
    &     & \ddots &      \\
    &     &        & B_k  \\
\end{pmatrix}
}
\end{equation} where each pair of the corresponding blocks $A_i$ and $B_i$  has one of the
the following types:

\begin{enumerate}
\item Jordan block with eigenvalue $\lambda \in
\mathbb{K}$ {\scriptsize  \[A_i =\left(
\begin{array}{c|c}
  0 & \begin{matrix}
   \lambda &1&        & \\
      & \lambda & \ddots &     \\
      &        & \ddots & 1  \\
      &        &        & \lambda   \\
    \end{matrix} \\
  \hline
  \begin{matrix}
  -\lambda  &        &   & \\
  -1   & -\lambda &     &\\
      & \ddots & \ddots &  \\
      &        & -1   & -\lambda \\
  \end{matrix} & 0
 \end{array}
 \right)
\quad  B_i= \left(
\begin{array}{c|c}
  0 & \begin{matrix}
    1 & &        & \\
      & 1 &  &     \\
      &        & \ddots &   \\
      &        &        & 1   \\
    \end{matrix} \\
  \hline
  \begin{matrix}
  -1  &        &   & \\
     & -1 &     &\\
      &  & \ddots &  \\
      &        &    & -1 \\
  \end{matrix} & 0
 \end{array}
 \right)
\]} \item Jordan block with eigenvalue $\infty$ {\scriptsize \[
A_i = \left(
\begin{array}{c|c}
  0 & \begin{matrix}
   1 & &        & \\
      &1 &  &     \\
      &        & \ddots &   \\
      &        &        & 1   \\
    \end{matrix} \\
  \hline
  \begin{matrix}
  -1  &        &   & \\
     & -1 &     &\\
      &  & \ddots &  \\
      &        &    & -1 \\
  \end{matrix} & 0
 \end{array}
 \right)
\quad B_i = \left(
\begin{array}{c|c}
  0 & \begin{matrix}
    0 & 1&        & \\
      & 0 & \ddots &     \\
      &        & \ddots & 1  \\
      &        &        & 0   \\
    \end{matrix} \\
  \hline
  \begin{matrix}
  0  &        &   & \\
  -1   & 0 &     &\\
      & \ddots & \ddots &  \\
      &        & -1   & 0 \\
  \end{matrix} & 0
 \end{array}
 \right)
 \] } \item  Kronecker block {\scriptsize \[ A_i = \left(
\begin{array}{c|c}
  0 & \begin{matrix}
   1 & 0      &        &     \\
      & \ddots & \ddots &     \\
      &        & 1    &  0  \\
    \end{matrix} \\
  \hline
  \begin{matrix}
  -1  &        &    \\
  0   & \ddots &    \\
      & \ddots & -1 \\
      &        & 0  \\
  \end{matrix} & 0
 \end{array}
 \right) \quad  B_i= \left(
\begin{array}{c|c}
  0 & \begin{matrix}
    0 & 1      &        &     \\
      & \ddots & \ddots &     \\
      &        &   0    & 1  \\
    \end{matrix} \\
  \hline
  \begin{matrix}
  0  &        &    \\
  -1   & \ddots &    \\
      & \ddots & 0 \\
      &        & -1  \\
  \end{matrix} & 0
 \end{array}
 \right)
 \] }
 \end{enumerate}

Each Kronecker block is a $(2k_i+1) \times (2k_i+1)$   block, where  $k_i \geq 0$.  In particular, if  $k_i=0$, then  $A_i$ and $B_i$ 
are two $1\times 1$ zero matrices
\[ A_i =
\begin{pmatrix}
0
\end{pmatrix} \quad  B_i=
\begin{pmatrix}
0
\end{pmatrix}
\]
\end{theorem}

There exists a natural real analogue of the Jordan--Kronecker theorem. 

\begin{theorem}[Real Jordan--Kronecker theorem]
\label{T:Real_Jordan_Kronecker} Any two skew-symmetric bilinear forms $A$ and $B$ on a real finite-dimensional vector
space $V$ can be reduced simultaneously to a block-diagonal form such that each block is either
a Kronecker block or a Jordan block with eigenvalue $\lambda \in \mathbb{R} \cup \{\infty\}$ or a real Jordan block with
complex eigenvalue $\lambda = \alpha + i \beta$:

{\scriptsize \[ A_i =\left(
\begin{array}{c|c}
  0 & \begin{matrix}
   \Lambda &E&        & \\
      & \Lambda & \ddots &     \\
      &        & \ddots & E  \\
      &        &        & \Lambda   \\
    \end{matrix} \\
  \hline
  \begin{matrix}
  - \Lambda  &        &   & \\
  - E   & - \Lambda &     &\\
      & \ddots & \ddots &  \\
      &        & - E   & - \Lambda \\
  \end{matrix} & 0
 \end{array}
 \right)
\quad  B_i= \left(
\begin{array}{c|c}
  0 & \begin{matrix}
    E & &        & \\
      & E &  &     \\
      &        & \ddots &   \\
      &        &        & E   \\
    \end{matrix} \\
  \hline
  \begin{matrix}
  - E  &        &   & \\
     & - E &     &\\
      &  & \ddots &  \\
      &        &    & - E \\
  \end{matrix} & 0
 \end{array}
 \right)
\]}

Here $\Lambda$ and $E$ are $2\times 2$ matrices: $\Lambda =
\begin{pmatrix} \alpha & -\beta \\ \beta & \alpha \\ \end{pmatrix}$ and $E = \begin{pmatrix} 1 & 0 \\ 0 & 1 \\ \end{pmatrix}
$. \end{theorem}

The proof of the Jordan--Kronecker theorem can be found in  \cite{Gurevich50} and \cite{Thompson91} (note that the proof in  \cite{Thompson91} is based on the results from the book  \cite{Gantmaher88}). For more details about the Jordan--Kronecker theorem see also  \cite{KozlovIK13}.

The pair of block-diagonal matrices \eqref{Eq:Jordan-Kronecker-form} consisting of Jordan and Kronecker blocks will be called the
Jordan--Kronecker form of the pair of forms $A$ and $B$. The Jordan--Kronecker form of the pair of forms $A$ and $B$ is uniquely defined up to a permutation of blocks.

If one of the forms is non-degenerate, then there is no Kronecker blocks in the Jordan--Kronecker form.

\begin{definition}
Let $(\omega_0,\omega_1)$ be a pair of compatible  $2$-forms on a
manifold $M$. A point $x_0 \in M$ will be called \emph{regular} if in some of its  neighbourhood $Ox_0$  the following Jordan--Kronecker invariants 
are constant:

\begin{itemize}

\item the number of distinct eigenvalues  $\lambda_i$ of the operators \[P_x : T_x M \to T_x M
,\]

\item the number and sizes of Jordan blocks for each eigenvalue $\lambda_i$.
\end{itemize} \end{definition}

Regular points can also be described as follows. Recall that the local frame is a set of vector fields defined in a  neighbourhood of a point of a manifold that are linearly independent at each point of this  neighbourhood.

\begin{remark} A point $x_0 \in M$  is regular if and only if in a  neighbourhood of this point there exists a local frame $v_1(x), \dots,
v_n(x)$  such that the matrices of both forms $\omega_0$ and $\omega_1$ have the block-diagonal form as in the  Jordan--Kronecker theorem \[{\footnotesize A =
\begin{pmatrix}
A_1 &     &        &      \\
    & A_2 &        &      \\
    &     & \ddots &      \\
    &     &        & A_k  \\
\end{pmatrix},
\quad  B=
\begin{pmatrix}
B_1 &     &        &      \\
    & B_2 &        &      \\
    &     & \ddots &      \\
    &     &        & B_k  \\
\end{pmatrix},
}\] where each eigenvalue $\lambda_i(x)$ depends on the point of the manifold: \begin{equation} \renewcommand*{\arraystretch}{1.2} A_i = \left(\begin{array}{c|c} 0 & J(\lambda_i(x)) \\
\hline - J^{T}(\lambda_i(x)) &
0\end{array} \right), \quad  B_i = \left(\begin{array}{c|c} 0 & E \\
\hline - E& 0\end{array} \right).\end{equation} \end{remark}

We now state the main results. The following two theorems allow us to reduce the problem of integration of invariant distributions to the case of one eigenvalue in the complex case or to the case of one eigenvalue or two complex conjugate eigenvalues in the real case.

\begin{theorem}[F.~Turiel, \cite{Turiel94}]\label{T:Eigenvalue_Strong_Decomposition}
Let  $(\omega_0, \omega_1)$  be a pair of compatible $2$-forms on a  manifold $M$. Then any regular point $x_0 \in (M, \omega_0, \omega_1)$  has a  neighbourhood $Ox$ isomorphic to a direct product of manifolds equipped with pairs of compatible $2$-forms:
\[(Ox, \omega_0, \omega_1) =  \prod (O_i x, \omega_{0, i}, \omega_{1, i}), \] where the characteristic polynomial of each pair of forms  $(\omega_{0, i}, \omega_{1, i})$ can not be decomposed into a product of two non-trivial
mutually prime polynomials. \end{theorem}

\begin{theorem}\label{T:InvDistEigenDecomp}
Suppose that a manifold $M$ with a pair of compatible  $2$-forms $(\omega_0, \omega_1)$ on it decomposes into a direct product 
\[(M, \omega_0, \omega_1) = (M', \omega_0', \omega_1') \times (M'', \omega_0'', \omega_1''),\] where $(\omega_0', \omega_1')$ and $(\omega_0'', \omega_1'')$ are pairs of compatible of  $2$-forms on $M'$ and $M''$ respectively. If the characteristic polynomials $\chi'$ and $\chi''$ of the pairs of forms $(\omega_0', \omega_1')$ and $(\omega_0'', \omega_1'')$ are mutually prime (at each point of the manifold), then any invariant distribution $F$ on $M$ is a direct product of  invariant distributions $F'$ and $F''$ on $M'$ and $M''$ respectively. Also, the invariant distribution $F$ on $M$ is integrable if only if the corresponding distributions $F'$ on $M'$ and $F''$ on $M''$ are integrable.
\end{theorem}

The case of one and two complex conjugate eigenvalues is described by the following theorem.

\begin{theorem} \label{T:IntInvSubs_NonConstEigen}
Consider a pair of compatible forms  $\omega_0$, $\omega_1$ on $M$ with
one eigenvalue $\lambda$ or with one pair of complex conjugate eigenvalues $\alpha \pm i \beta$. Let $x_0$ be a regular point of  $M$. Suppose that at the point $x_0$ the corresponding Jordan--Kronecker decomposition of $(T_{x_0}M, \omega_0, \omega_1)$ consists of Jordan  $k_1 +1, k_2, \dots, k_n$-blocks, where $k_1 \geq k_2 \geq \dots \geq k_n$. Then there exists a  neighbourhood of the point $x_0$, in which all invariant distribution, except, maybe for $\textnormal{Ker} (P - \lambda E)^{k_i}$ and $\textnormal{Ker} (P^2 - 2 \alpha P + (\alpha^2 + \beta^2)E)^{k_i}$, where $i>1$, are integrable.  In the case of one eigenvalue $\lambda$ the distributions $\textnormal{Ker} (P - \lambda E)^{k_i}$, where $i>1$, are integrable at $x_0$ if and only if the eigenvalue $\lambda$ is constant in a  neighbourhood of the point  $x_0$. Similarly, in the case of a pair of complex conjugate eigenvalues $\alpha \pm i \beta$  the distributions $\textnormal{Ker} (P^2 - 2 \alpha P + (\alpha^2 + \beta^2)E)^{k_i}$, where $i>1$, are integrable at  the point  $x_0$ if and only if the functions $\alpha$ and $\beta$ are constant in a  neighbourhood of the point $x_0$.
\end{theorem}

Together with the description of all linear invariant subspaces in Section  \ref{S:LinearInvSubs} (see Theorems \ref{T:Linear_Inv_Decomp}, \ref{T:InvSubsJord} and \ref{T:InvSubs_Real} as well as Remark \ref{Rem:InvSubsJord_Heights}) Theorems  \ref{T:Eigenvalue_Strong_Decomposition}, \ref{T:InvDistEigenDecomp} and  \ref{T:IntInvSubs_NonConstEigen} completely describe all integrable invariant distributions in the  neighbourhood of a regular point.

The proofs of Theorems  \ref{T:InvDistEigenDecomp} and
\ref{T:IntInvSubs_NonConstEigen} 
are given in section \ref{SubS:Proof_Main_Inv_Fol}.

\section{Linear invariant subspaces} \label{S:LinearInvSubs}

In this section, we describe all invariant subspaces of a linear space $V$ equipped with two skew-symmetric bilinear form $A$ and $B$, one of which is non-degenerate (we will assume that  $B$ is the non-degenerate form). In this case, instead of a pair of bilinear
forms $A$ and $B$ it will be more convenient to consider a pair consisting of the non-degenerate bilinear form $B$ and the operator $P = B^{-1}A$ which is selfadjoint with respect to the form $B$.

All theorems about the structure of linear invariant subspaces
immediately follow from the structure of the group of automorphisms $\textnormal{Aut}(V, B,
P)$, which was described by P.~Zhang in \cite{Pumei10}.

\textbf{Reduction of the problem.} First, we reduce the problem to the case when
the characteristic polynomial of the operator  $P$  is irreducible. Recall that  the  generalized eigenspace $V^{\lambda}$  of an operator $P$ corresponding to an eigenvalue $\lambda$ is the set $\textnormal{Ker}(P - \lambda E)^N$, where the number $N$  is sufficiently large. In the real case we will consider either generalized eigenspaces with real eigenvalues or generalized eigenspaces corresponding to a pair of complex conjugate eigenspace. By the generalized eigenspace corresponding to a pair of complex conjugate eigenvalues $\alpha \pm i \beta$ we mean the kernel \[\textnormal{Ker} (P^2 - 2 \alpha P + (\alpha^2 + \beta^2)E)^N.\]

Note that since the operator $P$ is self-adjoint, the generalized eigenspaces corresponding to different eigenvalues are
orthogonal with respect to the form $B$. Therefore the restriction of the form $B$ to each generalized eigenspace is non-degenerate.

\begin{theorem} \label{T:Linear_Inv_Decomp}
Let  $P$ be a linear self-adjoint operator on a real or complex symplectic space $(V, B)$. Consider the decomposition of the space $V$  into the sum of generalized eigenspaces of the operator $P$: \[V = \bigoplus_{\lambda} V^{\lambda} \]  Then a subspace $W \subset V$ is invariant if and only if it decomposes into a direct sum of its intersections with the generalized eigenspaces \[W = \bigoplus_{\lambda} (W \cap V^{\lambda} )\] and each intersection $W \cap V^{\lambda}$ is an invariant
subspace of the corresponding generalized eigenspace  $(V^{\lambda}, B|_{V^{\lambda}}, P|_{V^{\lambda}})$.
\end{theorem}

To prove Theorem \ref{T:Linear_Inv_Decomp} it suffices to use the fact that a subspace of the space $V$ is invariant if and only if it is invariant under the action of the groups $\textnormal{Aut}(V^{\lambda}, B|_{V^{\lambda}}, P|_{V^{\lambda}})$, which are naturally embedded in $\textnormal{Aut}(V, B, P)$.

From now on we will assume that the characteristic polynomial of the operator  $P$ is irreducible. This means that the operator $P$ has only one eigenvalue in the complex case and either one real or a pair of complex conjugate eigenvalues in the real case.

\textbf{The case of one eigenvalue.}  We first consider the case when there is only one eigenvalue. Without loss of generality, we can assume that this eigenvalue is $0$ (in other words, we can assume that the operator  $P$ is nilpotent).

\begin{theorem} \label{T:InvSubsJord} Let $P$ be a nilpotent self-adjoint operator on a symplectic space $(V, B)$. Then any subspace $W \subset V$ which is invariant under the action of the group of automorphisms $\textnormal{Aut}(V, B, P)$ has the form  \begin{equation}
\label{Eq:InvSubsJord_ImKer} W = \bigoplus_{i=1}^s (\textnormal{Ker} P^{k_i}
\cap \textnormal{Im} P^{l_i}), \end{equation} for some $s \in \mathbb{N}$,
$k_i, l_i \geq 0$.   \end{theorem}

\begin{corollary} \label{Cor:InvSubsJord_OneJord}
If a space  $(V, B, P)$  is the sum of Jordan blocks of the same height $k$, then the invariant subspaces are precisely the spaces \[\textnormal{Ker} P^i = \textnormal{Im} P^{k-i}\]
\end{corollary}

The following remark allows us to impose restrictions on the number and types of terms in the formula \eqref{Eq:InvSubsJord_ImKer}.

\begin{remark}\label{Rem:InvSubsJord_Heights}
The sum of subspaces of the form $\textnormal{Ker} P^{k_i} \cap \textnormal{Im} P^{l_i}$  can be described as follows. We denote by $V_J^k$ the sum of all Jordan blocks of the height (exactly) $k$.  Then  by $U^m(V_J^k)$  we denote the subspace of $V_J^k$ consisting of all vectors of the height $m$.

Let the space  $(V, B, P)$ be the sum of Jordan blocks of the height  $k_1, \dots, k_N$, where $k_1 > k_2 > \dots > k_N$. Then any
invariant subspace has the form \begin{equation} \label{Eq:InvSubsJord_Heights} U^{m_1}(V_J^{k_1}) \oplus \dots \oplus U^{m_N}
(V_J^{k_N}),\end{equation} where  \begin{equation}
\label{Eq:InvSubsJord_Heights_Conditions}
\begin{gathered} 0 \leq m_1 - m_2 \leq k_1 - k_2 \\ \cdots \\ 0 \leq m_{N-1} - m_N \leq k_{N-1} - k_N
\end{gathered}\end{equation} \end{remark}

Theorem \ref{T:InvSubsJord} easily follows from the following theorem about the structure of the Lie algebra of the automorphism group $\textnormal{Aut}(V, B, P)$  in the case when there is only one eigenvalue, so we will only sketch  a possible proof.

\begin{theorem}[P.~Zhang, \cite{Pumei10}]\label{T:BiSymp_General_Jordan_Case}
Suppose that the space $(V, A, B)$ consists of $l_i$ Jordan
$k_i$-blocks with eigenvalue $0$, where $i=1, \dots, N$ and $k_1
> k_2 > \dots > k_N$. Consider a basis of $V$  in which the matrices of the operator 
$P=B^{-1}A$ and the form $B$ are block-diagonal
\begin{equation}\label{E:Basis_Jordan_General_Case}
P = \begin{pmatrix} P_1 & & \\ & \ddots & \\ & & P_N  \end{pmatrix},
\qquad  B =  \begin{pmatrix} B_1 & & \\ & \ddots & \\ & & B_N
\end{pmatrix},
\end{equation}  where the blocks $P_i$ and $B_i$ are equal to \begin{equation} P_i = \left( \begin{array}{cccc} 0 &
& & \\ E_{2l_i} & \ddots & & \\ & \ddots & \ddots & \\
& & E_{2l_i} \end{array} \right) , \quad \text{and} \quad B_i = \left(
\begin{array}{cccc}  & & & Q_{2l_i} \\ & & \udots &  \\ & \udots &  &  \\
Q_{2l_i} & &  &  \end{array} \right),  \end{equation} where $Q_{2s} =
\begin{pmatrix} 0 & E_s \\ - E_s & 0 \end{pmatrix}$ and $E_s$ is the $s\times s$ identity matrix.  Then the Lie algebra of the automorphism group $\textnormal{Aut}(V, B, P)$ consists of the elements of the form \begin{equation} \label{E:bsp_algebra_matrix}
\left( \begin{array}{ccccc|ccc|c} C^{1,1}_1 & & & & & &  & & \multirow{5}{*}{$\cdots$} \\
C^{1,1}_2 & C^{1,1}_1 & & & & & & &  \\ \vdots & \ddots & \ddots &  & & C^{1,2}_1 & & & \\
\vdots & \ddots & \ddots & \ddots & & \vdots & \ddots & &
\\ C^{1,1}_n & \cdots & \cdots & \cdots & C^{1,1}_1  & C^{1,2}_m & \cdots & C^{1,2}_1 &
\\ \hline C^{2,1}_1 &  & & & & C^{2,2}_1 & &  & \multirow{3}{*}{$\cdots$} \\
\vdots & \ddots & & & & \vdots & \ddots  & & \\ C^{2,1}_m & \cdots &
C^{2,1}_1 & & & C^{2,2}_m & \cdots & C^{2,2}_1 & \\ \hline
\multicolumn{5}{c|}{\cdots}   & \multicolumn{3}{c|}{\cdots} & \cdots
\end{array} \right),
\end{equation}
where $C^{i,j}_s$ are  $l_i \times l_j$ matrices satisfying the relations
\begin{equation} \label{E:Cond_on_BiSymp_Jordan_Case} (C^{j,i}_s)^T Q_{2l_j} + Q_{2l_i} C^{i, j}_s =0
\end{equation}
In particular, the elements of the diagonal blocks $C^{i,i}_s$ are elements of the symplectic Lie algebra $sp(2l_i, \mathbb{K})$ since \begin{equation} (C^{i,i}_s)^{T} Q_{2l_i} + Q_{2l_i} C^{i, i}_s =0
\end{equation}
\end{theorem}

\begin{proof}[The scheme of the proof of Theorem \ref{T:InvSubsJord}] It is obvious that any subspace of the form \eqref{Eq:InvSubsJord_ImKer}  has the form \eqref{Eq:InvSubsJord_Heights}  and vice versa. For example, \[
U^{m_1}(V_J^{k_1}) \oplus \dots \oplus U^{m_N} (V_J^{k_N}) = \bigoplus^{N}_{i=1} (\textnormal{Ker} P^{m_i} \cap \textnormal{Im} P^{k_i - m_i}) \]  It is also obvious that the subspaces of the form \eqref{Eq:InvSubsJord_ImKer} are invariant. Therefore, it remains to show that any invariant subspace has the form \eqref{Eq:InvSubsJord_Heights}. Further, if we fix a decomposition \[V = V_J^{k_1} \oplus \dots \oplus V_J^{k_N},\] then we get a natural embedding  $\textnormal{Aut}(V_J^{k_i}) \hookrightarrow \textnormal{Aut}(V, B, P)$ for which the elements of $\textnormal{Aut}(V_J^{k_i})$  act identically on all terms $V_J^{k_j}$ except $V_J^{k_i}$.  Therefore, any invariant subspace  $W$ is a direct sum of invariant subspaces \[W = \bigoplus(W \cap V_J^{k_i}) \] Thus, the problem has been reduced to the case when there are only one or two types of Jordan blocks. First of all, using Theorem \ref{T:BiSymp_General_Jordan_Case}, we can prove Corollary \ref{Cor:InvSubsJord_OneJord}. Thus it will be proved that the invariant subspaces of  $V_J^{k_i}$  are the subspaces $U^m(V_J^{k_i})$  and only them. After that it remains to
prove the inequalities \eqref{Eq:InvSubsJord_Heights_Conditions}. They follow from the following two statements, which also can be easily proved using Theorem \ref{T:BiSymp_General_Jordan_Case}.

\begin{assertion}
If there exists a vector $v \in V_J^{k_1} \cap \textnormal{Im} P^l$ in an invariant subspace $W \subset (V, B, P)$, then $V_J^{k_2} \cap \textnormal{Im}
P^{l} \subset W$ for any $k_2< k_1$.
\end{assertion}

\begin{assertion}
If there exists a vector $v \in V_J^{k_2} \cap \textnormal{Ker} P^k$  in an invariant subspace $W \subset (V, B, P)$, then  $V_J^{k_1} \cap \textnormal{Ker}
P^{k} \subset W$ for any $k_2< k_1$.
\end{assertion}
\end{proof}

\textbf{Real case.} Let us now consider the case of two complex conjugate eigenvalues.   This case is easy reduced to the complex case due to existence of a natural complex structure.

\begin{lemma}\label{L:Complex_Str_in_Real_Jordan_Eigenspace}
In each generalized eigenspace of the operator $P$ corresponding to the pair of complex
conjugate eigenvalue $\alpha \pm i\beta$  the semisimple part of the operator $\frac{P - \alpha E }{\beta}$
is a complex structure $J$ which is self-adjoint with respect to $A$ and $B$.
\end{lemma}

\begin{proof}[Proof of Lemma \ref{L:Complex_Str_in_Real_Jordan_Eigenspace}]
$J$ is a complex structure, i.e. $J^2= - E$, since   the characteristic polynomial of the operator
 $\frac{P - \alpha E }{\beta}$ is equal to $(t^2 +1)^n$.   The operator $J$  is self-adjoint since it is a polynomial in $P$. \end{proof}

Suppose that the space $(V, A, B)$ consists only of real Jordan blocks with the same complex eigenvalue $\lambda = \alpha + i \beta$. Denote by $V^{\mathbb{C}}$ the complexification of the space  $V$ with respect to the complex structure $J$ from Lemma \ref{L:Complex_Str_in_Real_Jordan_Eigenspace}.  Consider the following complex bilinear forms on $V$:
\[A^{\mathbb C}(u, v) = A(u, v) - i A(u, Jv), \quad \text{and} \quad
B^{\mathbb C}(u, v) = B(u, v) - i B(u, Jv). \] The forms $A^{\mathbb C}$ and $B^{\mathbb C}$  are well-defined complex
bilinear forms on $V$ since the operator $J$ is  self-adjoint with respect to $A$ and $B$.

Note that the space  $V^{\mathbb{C}}$  consists of the same blocks as the space $V$, i.e. to every real Jordan $k$-block with eigenvalue $\alpha + i \beta$  in the decomposition of $V$ there corresponds a complex Jordan $k$-block with the same eigenvalue in the decomposition of the space $V^{\mathbb{C}}$.

This complexification allows us to identify the automorphisms of  $(V, A, B)$ with the automorphisms of $(V^{\mathbb{C}}, A^{\mathbb{C}}, B^{\mathbb{C}})$ and the real invariant subspaces with the complex invariant subspaces.

\begin{theorem} \label{T:InvSubs_Real}
Suppose that the space $(V, A, B)$ consists of only real Jordan blocks with the same complex eigenvalue $\lambda = \alpha + i \beta$.  Then every automorphism $Q \in \textnormal{Aut}(V, A, B)$   preserves the natural complex structure $J$ from  Lemma \ref{L:Complex_Str_in_Real_Jordan_Eigenspace}: \[  Q J = J Q.\] As a consequence, there is a natural one-to-one correspondence \begin{equation} \label{E:BiSymp_Real_Complex} \textnormal{Aut}(V, A, B) \cong \textnormal{Aut}(V^{\mathbb{C}}, A^{\mathbb{C}},
B^{\mathbb{C}}). \end{equation}  The subspace $U \subset (V, A, B)$  is invariant if and only if it is $J$-invariant and the corresponding subspace $U^{\mathbb{C}}$ of the space $(V^{\mathbb{C}}, A^{\mathbb{C}}, B^{\mathbb{C}})$  is invariant.
\end{theorem}

\begin{proof} The automorphism $Q \in \textnormal{Aut}(V, A, B)$ preserves the structure $J$ because it  preserves the forms $A$ and $B$. In order to prove that the invariant subspace can be complexified we can use the fact that if an invariant subspace $U$  contains a  vector $u$, then it also contains all (finite) linear combinations of vectors of the form $Qu$, where
$Q \in \textnormal{Aut}(V, A, B)$. To prove that if $u \in U$, then $Ju \in U$ it is sufficient to fix an arbitrary Jordan--Kronecker form and consider operators which act on one real Jordan block by the matrices \[\begin{pmatrix} R^T & 0 \\ 0 & R^{-1}\end{pmatrix}, \qquad \text{where} \qquad R={\small \left(
\begin{matrix}
\begin{array}{|c c|}\hline c & -d
\\ d & c \\
\hline \end{array}
  & &         \\
        &        \ddots &  \\
    &  & \begin{array}{|c c|} \hline c & -d \\ d & c \\ \hline  \end{array} \\
\end{matrix} \right) },
\qquad c,d \in \mathbb{R}, \quad c^2 + d^2 \ne 0\] and act trivially on the other Jordan blocks. \end{proof}

\section{Local properties of non-degenerate bi-Hamiltonian structures}

In this section we describe the main results proved by Turiel in \cite{Turiel94}, which we reformulate in a convenient for us form. In the paper \cite{Turiel94} the structure of compatible $2$-forms is described  in the  neighbourhoods of regular points $x_0 \in (M, \omega_0, \omega_1)$ such that each eigenvalue is either constant or has no critical points in a  neighbourhood of the point $x_0$:
\[\lambda_i (x) \equiv \textnormal{const} \qquad \text{or} \qquad d \lambda_i (x)\ne 0.\] We will call such regular points \emph{regular uncritical}.

The following general theorem reduces the problem to the case when the characteristic polynomial of the field of endomorphisms $P = \omega_0^{-1}
\omega_1$ is irreducible.

\begin{theorem}[F.~Turiel, \cite{Turiel94}] \label{T:CompSymp_Characteristic_Decomposition}
Let  $(\omega_0, \omega_1)$ be a pair of compatible $2$-forms  on a manifol $M$.  Suppose that the characteristic polynomial  
$\chi(x)$ of the field of endomorphisms $P = \omega_0^{-1} \omega_1$ decomposes into a direct product of polynomials which are mutually prime at each point of the manifold $M$:
\[ \chi (x) = \chi_1 (x) \chi_2(x), \qquad \qquad
\textnormal{GCD}({\chi_1(x), \chi_2(x)}) \equiv 1.\] Then any point $x \in M$ has a  neighbourhood $Ox$ that can be decomposed into a direct product \[(Ox, \omega_0, \omega_1) = (O'x, \omega_0', \omega_1') \times  (O''x,
\omega_0'', \omega_1'')
\] so that the characteristic polynomials of the pairs of forms $(\omega_0', \omega_1')$ and $(\omega_0'', \omega_1'')$ are $\chi_1(x)$ and $\chi_2(x)$ respectively.
\end{theorem}

Note that Theorem \ref{T:CompSymp_Characteristic_Decomposition} holds in a  neighbourhood of any regular point.

The case of constant eigenvalues is described by the following theorem.

\begin{definition} We will call a non-degenerate bi-Hamiltonian structure $(\omega_0, \omega_1)$  \emph{flat} at a point $x \in M$ if there exist local coordinates in a  neighbourhood of the point in which the matrices of both forms $\omega_0$ and $\omega_1$ have constant  coefficients. \end{definition}

\begin{theorem}[F.~Turiel, \cite{Turiel94}] A non-degenerate bi-Hamiltonian structure $(\omega_0, \omega_1)$  is flat in a  neighbourhood of a point $x\in M$ if and only if the point $x$ is regular and all the eigenvalues are constant. \end{theorem}

The case of one eignenvalue without critical points is described by the following Theorems \ref{T:Turiel_OneNonConst_Isom} and
\ref{T:Turiel_OneNonConst_JK_Inv}.

\begin{definition}
Two pairs of non-degenerate bi-Hamiltonian structures $(M', \omega'_0, \omega'_1)$ and $(M'', \omega''_0, \omega''_1)$ will be called \emph{ isomorphic}  if there exists a diffeomorphism   $f: M' \to M''$ such that  $f^* \omega''_0 = \omega'_0$ and $f^* \omega''_1 = \omega'_1$.
\end{definition}

\begin{theorem}[F.~Turiel, \cite{Turiel94}]
\label{T:Turiel_OneNonConst_Isom}Consider two pairs of compatible $2$-forms  $(\omega'_0, \omega'_1)$ and $(\omega''_0, \omega''_1)$ on manifolds $M'$ and $M''$ such that each pair has only one non-constant eigenvalue without critical points.
Then regular points $x' \in M'$ and $x'' \in M''$ have isomorphic  neighbourhoods if and only the Jordan--Kronecker decompositions of the pairs of forms $(\omega'_0, \omega'_1)$ and $(\omega''_0, \omega''_1)$  at these points coincide (i.e., consist of the same set of Jordan blocks). \end{theorem}

\begin{theorem}[F.~Turiel, \cite{Turiel94}] \label{T:Turiel_OneNonConst_JK_Inv}
Let $(\omega_0, \omega_1)$ be a pair of compatible $2$-forms  on a manifold $M$ with one eigenvalue $\lambda$  which has no critical points on $M$. Then for any regular point $x_0 \in M$ the biggest Jordan block in the Jordan--Kronecker decomposition of the corresponding tangent space $(T_{x_0}M, \omega_0, \omega_1)$ is always (strictly) bigger than other Jordan blocks.
\end{theorem}

Let us now describe a non-degenerate bi-Hamiltonian structure for each set of Jordan blocks with only one largest Jordan block.

\begin{theorem}[F.~Turiel, \cite{Turiel94}] \label{T:Turiel_OneNonConst_LocalForm}
Let $\omega_0$ and $\omega_1$ be compatible $2$-forms on a manifol $M$ and $x_0 \in M$ be a regular point.
Suppose that the field of endomorphisms $P = \omega_0^{-1} \omega_1$ has only one
eigenvalue $\lambda$ and $d\lambda|_{x_0} \ne 0$. Further, suppose that 
the Jordan--Kronecker decomposition of the pair of forms  $\omega_0$ and $\omega_1$ at the point $x_0$ consists of Jordan $k_1 +1, k_2, \dots, k_n$ blocks, where $k_1 \geq k_2 \geq \dots \geq k_n$.  Then in a  neighbourhood of the point $x_0$ there exist local coordinates \[(x_1^1, \dots, x_1^{k_1},
y_1^1, \dots, y_1^{k_1}, x_2^1, \dots, \dots, y_n^{k_n}, z, \lambda)
\] such that \begin{equation} \begin{gathered} \omega_0 =
(\sum_{s=1}^n
\sum_{i=1}^{k_s} dx^i_s \wedge d y_s^i ) + dz \wedge d \lambda  \\
\omega_1 = \lambda \omega_0 + \left( \sum_{s=1}^n \sum_{i=1}^{k_s-1}
dx_s^i \wedge dy_s^{i+1} \right) + \alpha \wedge d\lambda + dy^1_1
\wedge d\lambda,
\end{gathered} \end{equation} where
\begin{equation} \alpha = \sum_{s=1}^n \sum_{i=1}^{k_s} (i+\frac{1}{2})y_s^i
dx^i_s + (i- \frac{1}{2}) x^i_s dy^i_s.\end{equation} \end{theorem}

In other words, the matrices of the forms are as follows:
\begin{equation} \begin{gathered}
\omega_0 = \left(%
\begin{array}{ccccc|cc} %
0 & E_{k_1} & & & & & \\ %
-E_{k_1} & 0 & & & & & \\ %
 &  & \ddots & & & &  \\ %
&  &  & 0 & E_{k_n} & & \\ %
&  &  & -E_{k_n} & 0 & & \\ %
\hline
&  &  & & & 0 & 1 \\ %
&  &  & & & -1 & 0 \\ %
\end{array} \right),\\ \omega_1 = \left(%
\begin{array}{ccccc|cc} %
0 & J_{k_1}(\lambda) & & & & 0 & \alpha_1 \\ %
-J^T_{k_1}(\lambda) & 0 & & & & 0 & \beta_1 +\delta \\ %
 &  & \ddots & & &  \vdots & \vdots  \\ %
&  &  & 0 & J_{k_n}(\lambda) & 0 & \alpha_n \\ %
&  &  & -J^T_{k_n}(\lambda) & 0 & 0 &  \beta_n \\ %
\hline
0 & 0 & \cdots &0 & 0 & 0 & \lambda \\ %
-\alpha_1^T &  -\beta_1^T - \delta^T & \cdots & -\alpha_n^T & -\beta_n^T & -\lambda & 0 \\ %
\end{array} \right),
\end{gathered} \end{equation}

where \begin{equation} \label{Eq:Turiel_FormsMatrixAddit} \alpha_s =
\left( \begin{matrix} \frac{3}{2} y_s^1
\\[0.3em] \frac{5}{2} y_s^2 \\[0.3em] \vdots \\[0.3em] (k_s + \frac{1}{2}) y_s^{k_s}
\end{matrix} \right), \qquad \beta_s = \left( \begin{matrix} \frac{1}{2} x_s^1
\\[0.3em] \frac{3}{2} x_s^2 \\[0.3em] \vdots \\[0.3em] (k_s - \frac{1}{2}) x_s^{k_s}
\end{matrix} \right), \qquad \delta = \left( \begin{matrix} 1
\\[0.3em] 0 \\[0.3em] \vdots \\[0.3em] 0
\end{matrix} \right)
\end{equation}

\begin{remark}

\begin{itemize}

\item Let us emphasize that the last coordinate is the eigenvalue $\lambda$.  Also note that $\frac{\partial}{\partial z}$ is the Hamiltonian vector field with Hamiltonian $\pm \lambda$ with respect to the form $\omega_0$.  (The sign $\pm$ depends on the sign convention in the definition of Hamiltonian vector fields.)

\item The term  $\alpha\wedge d \lambda$  makes the form $\omega_1$ closed.

\item The term  $d y^1_1 \wedge d\lambda$ is required for the Jordan--Kronecker invariants to be constant  in a  neighbourhood of the point $x_0$.
\end{itemize}
\end{remark}

\begin{corollary}
The field of endomorphisms $P=\omega_0^{-1} \omega_1$  is given by the formula \begin{equation} \label{T:Turiel_EndomField} \begin{gathered} P = \lambda E
+ \left( \sum_{s=1}^n \sum_{j=1}^{k_s} \left(
\frac{\partial}{\partial x_s^{j+1}} +
(j+\frac{1}{2})y_s^j\frac{\partial}{\partial z}
 \right) \otimes d x_s^j \right.  + \\
 \left. \left(\frac{\partial}{\partial y_s^{k_s -j}} + (k_s +\frac{1}{2}
-j)x_s^{k_s -j +1}\frac{\partial}{\partial z} + \delta^1_s
\delta^{k_s}_j \frac{\partial}{\partial z} \right)\otimes d y_s^{k_s
+1 -j} \right) + \\
\left( \sum_{s=1}^n \sum_{j=1}^{k_s} - ((j- \frac{1}{2}) x_s^{j} +
\delta_{s}^1 \delta_j^1) \frac{\partial}{\partial x_s^j} + (j+
\frac{1}{2} ) y_s^{j} \frac{\partial}{\partial y_s^j} \right)
\otimes d \lambda,
\end{gathered} \end{equation} where we formally assume $x_s^{i} = y_s^i =0$ if $i>k_s$ or $i\leq 0$. In other words, the matrix of this field of endomorphisms has the form  \begin{equation} \label{T:Turiel_EndomField_Matrix} P = \left(%
\begin{array}{ccccc|cc} %
J^T_{k_1}(\lambda) & 0 & & & & 0 & -\beta_1 - \delta \\ %
0 & J_{k_1}(\lambda)  & & & & 0 & \alpha_1 \\ %
 &  & \ddots & & &  \vdots & \vdots  \\ %
&  &  &  J^T_{k_n}(\lambda) & 0 & 0 & -\beta_n \\ %
&  &  & 0 & J_{k_n}(\lambda)  & 0 &  \alpha_n \\ %
\hline
\alpha_1^T &  \beta_1^T + \delta^T & \cdots & \alpha_n^T & \beta_n^T & \lambda & 0 \\ %
0 & 0 & \cdots &0 & 0 & 0 & \lambda \\ %
\end{array} \right),\end{equation} where the vectors $\alpha_i, \beta_i$ and  $\delta$ are  given by the formulas \eqref{Eq:Turiel_FormsMatrixAddit}.
\end{corollary}

It follows immediately from the formula \eqref{T:Turiel_EndomField} that 
\begin{equation} \begin{gathered} \label{Eq:NilOperator}
(P-\lambda E) \frac{\partial}{\partial x_s^j} = \frac{\partial}{\partial x_s^{j+1}} + (j+\frac{1}{2})y_s^j\frac{\partial}{\partial z} \\
(P-\lambda E) \frac{\partial}{\partial y_s^{k_s +1 -j}} =
\frac{\partial}{\partial y_s^{k_s -j}} +
(k_s +\frac{1}{2} -j)x_s^{k_s -j +1}\frac{\partial}{\partial z} + \delta^1_s \delta^{k_s}_j \frac{\partial}{\partial z} \\
(P - \lambda E) \frac{\partial}{\partial \lambda} =\sum_{s=1}^n
\left(\sum_{j=1}^{k_s} - ((j- \frac{1}{2}) x_s^{j} + \delta_{s}^1
\delta_j^1) \frac{\partial}{\partial x_s^j} + (j+ \frac{1}{2} )
y_s^{j} \frac{\partial}{\partial y_s^j} \right).
\end{gathered} \end{equation}
Using the formula \eqref{Eq:NilOperator} several times we get  \begin{equation} \label{Eq:NilOperator_Several}
(P-\lambda E)^p \frac{\partial}{\partial x_s^j} = (P-\lambda
E)\frac{\partial}{\partial x_s^{j+p-1}} = \frac{\partial}{\partial
x_s^{j+p}} + (j+p-\frac{1}{2})y_s^{j+p-1}\frac{\partial}{\partial z}
\end{equation} and  \begin{equation} \begin{gathered} (P-\lambda E)^p \frac{\partial}{\partial y_s^{k_s +1 -j}} =
(P-\lambda E) \frac{\partial}{\partial y_s^{k_s +2 -j -p}} = \\=
\frac{\partial}{\partial y_s^{k_s  +1 -j -p}} +
(k_s +\frac{3}{2} -j - p)x_s^{k_s +2-j -p}\frac{\partial}{\partial z} + \delta^1_s \delta^{k_s+1}_{j+p} \frac{\partial}{\partial z} \\
\end{gathered} \end{equation}

\textbf{Real case.} The case of one pair of complex conjugate eigenvalues is actually similar to the corresponding complex case since the almost complex structure from Lemma \ref{L:Complex_Str_in_Real_Jordan_Eigenspace} turns out to be integrable.

\begin{theorem}[F.~Turiel, \cite{Turiel94}] \label{T:Turiel_Real}
Let $(\omega_0, \omega_1)$ be a pair of compatible $2$-forms on a real manifold $M$ with one pair of complex conjugate
eigenvalues $\alpha \pm i \beta$.  Then the semisimple part $J$ of the operator $\frac{P - \alpha E }{\beta}$, where $P = \omega_0^{-1} \omega_1$, is a complex structure on $M$. Let $M^{\mathbb{C}}$ be the complexification of the manifold $M$ by the complex structure $J$. Then the forms $\omega_0^{\mathbb{C}}$ and $\omega_1^{\mathbb{C}}$  given by the formulas \begin{equation} \omega_0^{\mathbb C}(u, v) = \omega_0(u, v) - i \omega_0(u,
Jv), \quad \text{and} \quad \omega_1^{\mathbb C}(u, v) = \omega_1(u,
v) - i \omega_1(u, Jv) \end{equation}  are compatible holomorphic forms on $M^{\mathbb{C}}$. The operator $P$ commutes with the complex structure $J$, therefore it induces an operator $P^{\mathbb{C}}$ on $M^{\mathbb{C}}$. The equality $P^{\mathbb{C}}=(\omega_0^{\mathbb{C}})^{-1} \omega_1^{\mathbb{C}}$ holds, therefore the pair of forms $(\omega_0^{\mathbb{C}}, \omega_1^{\mathbb{C}})$
on $M^{\mathbb{C}}$ has only one eigenvalue $\lambda = \alpha + i \beta$.
\end{theorem}

\begin{remark} Under the conditions of Theorem \ref{T:Turiel_Real} if we take the complex structure  $-J$ instead of the complex
structure $J$ the complexification of the pair of forms will have the eigenvalue $\alpha - i \beta$. \end{remark}

\section{Proof of the main theorems}
\label{SubS:Proof_Main_Inv_Fol}

Theorem \ref{T:Eigenvalue_Strong_Decomposition} is a direct consequence of Theorem \ref{T:CompSymp_Characteristic_Decomposition}. Theorem \ref{T:InvDistEigenDecomp}  follows from Theorems
\ref{T:Eigenvalue_Strong_Decomposition} and \ref{T:Linear_Inv_Decomp}. Thus, it remains only to prove Theorem
\ref{T:IntInvSubs_NonConstEigen}. The case of a pair of complex conjugate eigenvalues can be reduced to the case of one eigenvalue using Theorems \ref{T:InvSubs_Real} and \ref{T:Turiel_Real}, therefore it suffices to prove Theorem \ref{T:IntInvSubs_NonConstEigen} only in the case when there is only one eigenvalue. This entire section is devoted to this.

Regular non-critical points form an open dense set, therefore it suffices to prove Theorem \ref{T:IntInvSubs_NonConstEigen}  just for them. The case when the eigenvalues are constant (i.e. when the non-degenerate bi-Hamiltonian structure is flat) is trivial. Thus, it remains to consider the case of a non-constant eigenvalue without critical points (i.e. the case of Theorem \ref{T:Turiel_OneNonConst_LocalForm}). The proof is by direct computation. Let us  explicitly describe each invariant distribution and directly check it involutivity (the formula for vector fields generating the distributions is given in Lemma \ref{L:InvFol_Formula}).

First, let us describe the local frame \begin{equation} \label{Eq:OneEigenNonConst_Basis}  e^0_1, e^1_1 \dots,
e^{k_1}, f_1^0, \dots, f_1^{k_1} , e_2^1, \dots, e_n^{k_n}, f_n^1,
\dots, f_n^{k_n}\end{equation} in a  neighbourhood of $x_0$ in which the matrices of the forms $\omega_0$ and $\omega_1$ consist of Jordan blocks. The basis for smaller blocks (i.e. Jordan $k_2$, \dots, $k_n$-blocks):
\begin{equation} \label{Eq:OneEigenNonConst_Basis_Small} \begin{gathered} e_s^1 = \frac{\partial}{\partial x_s^1} -
\frac{(k_s + \frac{1}{2}) y_s^{k_s}}{(\frac{1}{2} x_1^1 +1)}
\frac{\partial} {\partial y_1^{k_s}}, \\ e_s^i = (P - \lambda
E)^{i-1} e_s^1 = \frac{\partial}{\partial x_s^i} - \frac{(k_s +
\frac{1}{2}) y_s^{k_s}}{(\frac{1}{2} x_1^1 +1)} \frac{\partial}
{\partial
y_1^{k_s-i+1}} + \alpha_s^i \frac{\partial}{\partial z},\\
f_s^{k_s} = \frac{\partial}{\partial y_s^1} - \frac{\frac{1}{2}
x_s^1}{(\frac{1}{2} x_1^1 +1)} \frac{\partial} {\partial y_1^{k_s}},
\\ f_s^i = (P - \lambda E)^{k_s-i} f_s^{k_s} = \frac{\partial}{\partial
y_s^i} - \frac{\frac{1}{2} x_s^1}{(\frac{1}{2} x_1^1 +1)}
\frac{\partial} {\partial y_1^{i}} + \beta_s^i
\frac{\partial}{\partial z},
\end{gathered} \end{equation} where \begin{equation} \begin{gathered} \alpha_s^i = (i-\frac{1}{2}) y_s^{i-1}
- \frac{(k_s + \frac{1}{2}) y_s^{k_s}}{(\frac{1}{2} x_1^1 +1)} ((k_s
-i + \frac{3}{2})x_1^{k_s -i +2} + \delta_i^{k_s +1})
\\ \beta_s^i = (i + \frac{1}{2}) x_s^{i+1}
- \frac{\frac{1}{2} x_s^1}{(\frac{1}{2} x_1^1 +1)}
((i+\frac{1}{2})x_1^{i+1} + \delta_0^{i})
\end{gathered} \end{equation}

The basis for the biggest  $(k_1+1)$-block: \begin{equation}
\label{Eq:OneEigenNonConst_Basis_Big}
\begin{gathered} e_1^0 = -\frac{\partial}{\partial \lambda} +
\sum_{s=1}^n \left( \sum_{j=1}^{k_s} - (j+ \frac{1}{2}) x_s^{j+1}
\frac{\partial}{\partial x_s^j} + (j- \frac{1}{2} ) y_s^{j-1}
\frac{\partial}{\partial y_s^j} \right) \\
e_1^i = (P - \lambda E)^{i} e_1^0 =  \sum_{s=1}^n \left( \left(
\frac{1}{2} x^1_s + \delta_s^1 \right) \frac{\partial} {\partial
x_s^{i}} - \left( k_s + \frac{1}{2} \right) y_s^{k_s}
\frac{\partial} {\partial y_s^{k_s-i+1}}\right) + \gamma_i
\frac{\partial} {\partial z} \\
f_1^{k_1} =\frac{1}{(\frac{1}{2} x_1^1 +1)} \frac{\partial}
{\partial y_1^{k_1}},
\\
f_1^i = (P - \lambda E)^{k_1-i} f_1^{k_1} = \frac{1}{(\frac{1}{2}
x_1^1 +1)} \frac{\partial} {\partial y_1^{i}} + \beta_1^i
\frac{\partial}{\partial z}, \\ f_1^0 = \frac{\partial}{\partial z}
\end{gathered} \end{equation}where \begin{equation}
\begin{gathered}\gamma_i =  \sum_{s=1}^n (\frac{1}{2}x_s^1 +
\delta_s^1)(i-\frac{1}{2}) y_s^{i-1} - (k_s + \frac{1}{2}) y_s^{k_s}(k_s - i + \frac{3}{2}) x_s^{k_s - i+2} \\
\beta_1^i = \frac{1}{\frac{1}{2} x_1^1 +1} \left(
\left(i+\frac{1}{2} \right)x_1^{i+1} + \delta_0^{i} \right).
\end{gathered} \end{equation} Here we formally assume that $x_s^{i}
= y_s^i =0$ if $i>k_s$ or $i\leq 0$. In particular, $\gamma_1 =0$.

The vector fields \eqref{Eq:OneEigenNonConst_Basis_Small} and \eqref{Eq:OneEigenNonConst_Basis_Big}  were chosen so that the following assertion holds.

\begin{assertion}
Under the conditions of Theorem \ref{T:Turiel_OneNonConst_LocalForm} in the local
reference frame \eqref{Eq:OneEigenNonConst_Basis} given by the formulas
\eqref{Eq:OneEigenNonConst_Basis_Small} and
\eqref{Eq:OneEigenNonConst_Basis_Big} the matrices of the forms  $\omega_0$ and $\omega_1$ have a block diagonal form {\footnotesize
\begin{equation} \omega_1= \left(\begin{matrix}
A_1 &     &        &      \\
    & A_2 &        &      \\
    &     & \ddots &      \\
    &     &        & A_k  \\
\end{matrix} \right),
\qquad  \omega_0= \left(\begin{matrix}
B_1 &     &        &      \\
    & B_2 &        &      \\
    &     & \ddots &      \\
    &     &        & B_k  \\
\end{matrix} \right),\end{equation}} where $A_1, B_1$ form a Jordan $(k_{1}+1)$-block, and $A_i, B_i$ are Jordan $k_i$-blocks for $i>1$. \end{assertion}

The structure of invariant subspaces (see Remark \ref{Rem:InvSubsJord_Heights}) implies the following statement.

\begin{lemma} \label{L:InvFol_Formula}
Under the conditions of Theorem \ref{T:IntInvSubs_NonConstEigen}  in the basis 
\eqref{Eq:OneEigenNonConst_Basis}  given by the formulas
\eqref{Eq:OneEigenNonConst_Basis_Small} and \eqref{Eq:OneEigenNonConst_Basis_Big} the invariant distributions have the form\begin{equation}\label{Eq:InvFol_Formula}
\begin{gathered} U^{m_1 + 1} (V_J^{k_1 +1}) \oplus \dots \oplus
U^{m_N}(V_J^{k_N}) = \langle \frac{\partial}{\partial z},
\frac{\partial}{\partial y^1_1}, \dots, \frac{\partial}{\partial
y^{m_1}_1},
e_1^{k_1}, \dots, e_1^{k_1 - m_1} \rangle \oplus \\
\bigoplus_{s=2}^N \langle \frac{\partial}{\partial x_s^{k_s}},
\dots, \frac{\partial}{\partial x_s^{k_s -m_s +2}},
\frac{\partial}{\partial x_s^{k_s - m_s +1}} - \frac{(k_s +
\frac{1}{2}) y_s^{k_s}}{(\frac{1}{2} x_1^1 +1)}
\frac{\partial}{\partial y_1^{m_s}}, \\
\frac{\partial}{\partial y_s^{1}}, \dots, \frac{\partial}{\partial
y_s^{m_s -1}}, \frac{\partial}{\partial y_s^{m_s}} -
\frac{\frac{1}{2} x_s^{1}}{(\frac{1}{2} x_1^1 +1)}
\frac{\partial}{\partial y_1^{m_s}} \rangle.
\end{gathered} \end{equation}

\end{lemma}

\begin{proof}[Proof of Theorem \ref{T:IntInvSubs_NonConstEigen}]
To prove Theorem \ref{T:IntInvSubs_NonConstEigen} it remains to check involutive of the distribution \eqref{Eq:InvFol_Formula}. This is done by direct calculation --- all we need to do is to find the commutator of the vector fields that define the distribution \eqref{Eq:InvFol_Formula} and check whether they are tangent to this distribution.

The distributions $\textnormal{Ker}(P-\lambda E)^{k_i}$, where $i>1$,  are non-integrable since the commutator of the vector fields \begin{gather} u_s =
\frac{\partial}{\partial x_s^{k_s - m_s +1}} - \frac{(k_s +
\frac{1}{2}) y_s^{k_s}}{(\frac{1}{2} x_1^1 +1)}
\frac{\partial}{\partial y_1^{m_s}},\\ v_s= \frac{\partial}{\partial
y_s^{m_s}} - \frac{\frac{1}{2} x_s^{1}}{(\frac{1}{2} x_1^1 +1)}
\frac{\partial}{\partial y_1^{m_s}} \end{gather}  is not tangent to the distribution. Indeed, \begin{equation} [u_s, v_s] = \delta^{k_s}_{m_s}
\left(\frac{k_s}{(\frac{1}{2}x_1^1 + 1)}\frac{\partial}{\partial
y_1^{k_s}}\right).\end{equation}  It is not hard to verify that the vector fields $\frac{\partial } {\partial y_1^{m_s}}$  are tangent to the invariant distribution \eqref{Eq:InvFol_Formula}  if and only if $m_1 \geq m_s$. For the distributions $\textnormal{Ker}(P-\lambda E)^{k_i}$, where $i>1$, we have $m_1 = m_s-1<  m_s = k_s$. Therefore, in these cases the commutator $[u_s, v_s]$ is not tangent to $\textnormal{Ker}(P-\lambda E)^{k_i}$.

The proof that all invariant distributions except for $\textnormal{Ker}(P-\lambda E)^{k_i}$, where $i>1$, are integrable can be slightly simplified by using the following considerations. Note that \[ [e^m_1, \frac{\partial }{\partial x_s^i}] \in \langle
\frac{\partial }{\partial z} \rangle, \qquad [e^m_1, \frac{\partial
}{\partial y_s^{k_s -i}}] \in \langle \frac{\partial }{\partial z}
\rangle \] for $i>1$ and that the vector fields $\frac{\partial }{\partial z}$ belong to any invariant distribution (except for the trivial zero-dimensional distribution). Thus the only non-trivial commutators of vector fields from the formula \eqref{Eq:InvFol_Formula} which may not be tangent to the distribution are the pairwise commutators of the vector fields $e_1^0, e_1^1, u_s$ and
$v_s$.  It follows from the structure of invariant subspaces (see Remark \ref{Rem:InvSubsJord_Heights}) that the only invariant distribution containing the vector field $e_1^0$ coincides with the whole tangent bundle. It also follows from Remark \ref{Rem:InvSubsJord_Heights} that the only invariant distribution containing the vector field $e_1^1$ which is different from the whole tangent bundle and from $\textnormal{Ker}(P-\lambda E)^{k_i}$, where $i>2$, is the distribution $\textnormal{Im} (P-\lambda E)$. Let us prove separately the involutivity of this distribution.

\begin{assertion}
Under the conditions of Theorem \ref{T:Turiel_OneNonConst_LocalForm} the distribution $\textnormal{Im} (P-\lambda E)$  is involutive. \end{assertion}
\begin{proof}
Since $N_P =0$ , the following equality holds \begin{equation}
\begin{gathered} \left[(P-\lambda E)u, (P-\lambda E)v \right] =(P-\lambda E)
([u, Pv] + [Pu, v] - (P+\lambda E)[u, v]) +
\\ (\mathcal{L}_{(P-\lambda E) v} \lambda) u
-(\mathcal{L}_{(P-\lambda E) u} \lambda) v
\end{gathered} \end{equation} The distribution $\textnormal{Im} (P-\lambda E)$ is involutive because $\mathcal{L}_{(P-\lambda E)u} \lambda =0$ for any vector field $u$. \end{proof}

Thus Theorem \ref{T:IntInvSubs_NonConstEigen} is completely proved. \end{proof}

\end{document}